\documentclass[12pt]{amsart}
\usepackage{amsfonts,amstext}
\usepackage{amsmath,amsthm,amssymb}
\usepackage{verbatim} 

\newtheorem{lemma}{Lemma}
\newtheorem{theorem}[lemma]{Theorem}

\theoremstyle{definition}

\newtheorem{remark}[lemma]{Remark}

\input xy
\xyoption{all} 

\def\cvec#1#2{{\tiny \begin{pmatrix} #1 \\ #2 \end{pmatrix}}}

\def\d{\delta}

\def\s{\sigma}
\def\t{\tau}
\def\x{\chi}
\def\D{\Delta}

\def\cA{\mathcal{A}}
\def\A{\mathbb{A}}
\def\sK{\mathcal{K}}
\def\cK{\sK}

\def\sP{\mathcal{P}}
\def\cE{\mathcal{E}}
\def\cF{\mathcal{F}}
\def\cP{\sP}

\def\fm{\mathfrak{m}}

\def\Spec{\operatorname{Spec}}
\def\Hom{\operatorname{Hom}}
\def\map#1{{\buildrel #1 \over \lra}}
\def\ker{\operatorname{ker}}
\def\coker{\operatorname{coker}}

\def\ov#1{\overline{#1}}

\newcommand{\real}{\mathbb{R}}

\newcommand{\ints}{\mathbb{Z}}
\def\Z{\ints}

\newcommand{\lra}{\longrightarrow}
\newcommand{\Lra}{\Longrightarrow}

\newcommand{\isom}{\cong}

\newcommand{\inj}{\hookrightarrow}
\def\into{\inj}
\newcommand{\surj}{\twoheadrightarrow}
\def\onto{\surj}
\newcommand{\dual}{\vee}
\newcommand{\tensor}{\otimes}
\newcommand{\intersect}{\cap}

 \newcommand{\ra}{\rightarrow}

 \def\l{\left}
 
 \def\r{\right}
 
 \def\sV{\mathcal{V}}
 \def\cV{\mathcal{V}}

 \def\sU{\mathcal{U}}

\begin{document}


\title{The Equivariant {$K$}-theory of Toric Varieties}

\author{Suanne Au}
\address{Dept.\ of Mathematics, University of Nebraska - Lincoln,
  Lincoln, NE 68588, USA}
\email{xau@math.unl.edu}

\author{Mu-wan Huang}
\email{s-mhuang2@math.unl.edu}

\author{Mark E. Walker}
\thanks{Walker's research was partially supported by NSF grant DMS-0601666.}
\email{mwalker5@math.unl.edu}

\date{\today}

\begin{abstract}
This paper contains two results concerning the equivariant $K$-theory
of toric varieties. The first is a formula for the equivariant
$K$-groups of an arbitrary affine toric variety, generalizing the known
formula for smooth ones. In fact, this result is established in a more
general context, involving the $K$-theory of graded projective
modules. The second result is a new proof of a theorem due to
Vezzosi and Vistoli concerning the equivariant $K$-theory of smooth
(not necessarily affine) toric varieties.
\end{abstract}



\maketitle

\tableofcontents

\section{Introduction}

Let $k$ be a field,
suppose $U_\sigma$ is the affine toric $k$-variety  associated to a
strongly convex rational polyhedral cone $\sigma$ in Euclidean
$n$-space,
and let $T$ be the $n$-dimensional torus that acts on $U_\sigma$.
If $U_\sigma$  is
smooth, then there is an equivariant isomorphism
$U_\sigma \cong T_\sigma \times \A^r$, where $r = \dim(\sigma)$ and
$T_\sigma$ is the unique orbit of minimal dimension (namely, dimension
$n-r$).
Using basic properties of equivariant $K$-theory of smooth
varieties (see, for example, \cite{Merk}), ones obtains natural isomorphisms
\begin{equation} \label{E1}
K_q^T(U_\sigma) \cong K_q^T(T_\sigma) \cong K_q(k) \otimes_\Z
\Z[M_\sigma]
\end{equation}
where $M_\sigma \cong \Z^{n-r}$ is the group of characters of $T_\sigma$.

This paper consists of two main results related to the isomorphism
\eqref{E1}. The first, Theorem \ref{Thm2}, shows that
this isomorphism holds for all affine toric varieties, not
just smooth ones.  In fact, this theorem establishes the
more general isomorphism
\begin{equation} \label{E7}
K_q^T(U_\sigma \times_k \Spec R) \cong K_q(R) \otimes_\Z
\Z[M_\sigma],
\end{equation}
where $R$ is any  $k$-algebra and the action of $T$ on
$\Spec R$ is trivial.
Theorem \ref{Thm2} is actually a consequence of our
Theorem \ref{MainThm1}, concerning the $K$-theory of graded projective modules.

The second main result of this paper is a new proof
of a theorem due to Vezzosi and Vistoli \cite[Theorem 6.2]{VV} that
calculates the equivariant $K$-theory of an arbitrary smooth
toric variety.  See our Theorem \ref{ExactSeq} for the
precise statement. The proof due to Vezzosi and Vistoli uses a more
general result, one that applies to arbitrary actions
by diagonalizable groups schemes. However, in the important special
case of toric varieties, we recover their result using only Equation
\eqref{E1}, the theory of sheaf cohomology for fans, and Thomason's
foundational work on equivariant $K$-theory \cite{Thomason}.

\section{The $K$-theory of graded projective modules}

The first main goal of this paper is to establish the isomorphism
\eqref{E7}.
The action of $T$ on $U_\sigma$ is given by a grading (by the group of
characters of $T$) of
the associated ring of regular functions for $U_\sigma$,
and an equivariant bundle on
$U_\sigma$ is given by a graded projective module over this ring. Thus, our first
theorem is really about the $K$-theory of graded projective
modules. In this section, we state and prove a general theorem of this
form.

Let $R$ be any commutative ring, $M$ an abelian group (written
additively), and $A \subset M$ a sub-monoid. We form the associated
monoid-ring $R[A]$. As a matter of notation, an element $a \in A$ is
written as $\x^a$ in $R[A]$ so that $\x^a \x^b = \x^{a+b}$ for $a,b
\in A$.  The commutative ring $R[A]$ is an $M$-graded $R$-algebra,
with elements of $R$ declared to be of degree zero and for any $a \in
A$, $\deg(\chi^a) := a \in A \subset M$.  Let $\cP(R)$ denote the
category of finitely generated projective $R$-modules and let
$\cP^M(R[A])$ denote the category consisting of 
finitely generated $M$-graded projective $R[A]$-modules and 
with morphisms
given by $M$-graded $R[A]$-module homomorphisms.
Let $K_*^{M}(R[A])$ denote
the $K$-theory of the exact category $\cP^M(R[A])$.

Recall that if $G$ is an $M$-graded $R[A]$-module and $m \in M$, then
$G[m]$ denotes the same module but with the grading shifted so that
$G[m]_{w} = G_{w-m}$ for all $w \in M$. In particular, $R[A][m]$ is graded-free of rank
one generated by an element of degree $m$.

Write $U(A)$ for the subgroup of units (i.e., elements with additive
inverses) in the monoid $A$. We fix, once and for all, a set 
$S(A) \subset M$ of coset representatives for the subgroup $U(A)$ of
$M$.

\begin{theorem} \label{MainThm1}
For a commutative ring $R$, an  abelian
  group  $M$, and a sub-monoid $A$ of $M$,
we have an isomorphism
$$
K_q(R) \otimes_\Z \Z[M/U(A)] \cong
K_q^{M}(R[A]), \, \, \text{for all $q$.}
$$
Under the identification of $K_q(R) \otimes_\Z \Z[M/U(A)]$ with
$\bigoplus_{S(A)} K_q(R)$, this isomorphism is 
induced by the collection of exact functors sending $(P, s)$,
with $P \in \cP(R)$ and $s \in S(A)$, to $P \otimes_R R[A][s]$.
\end{theorem}

The proof of the Theorem requires the following two lemmas. 
Throughout the rest of this section, let $U = U(A)$
and $S = S(A)$.

\begin{lemma}\label{EqSplitA} 
The exact functor
$$
\psi: \bigoplus_S \cP(R) \to \cP^M(R[U])
$$
determined by 
$$
(P_s)_{s \in S} \mapsto \bigoplus_{s \in S} P_s
\otimes_R R[U][s]
$$
is an equivalence of categories. 
\end{lemma}

\begin{proof}
For $P, P' \in \cP(R)$ and $s,s' \in S$, we have an isomorphism
{\small
\begin{equation} \label{E814}
\Hom_{R[U]}^M(P \otimes_R R[U][s], P' \otimes_R R[U][s'])
\cong
\begin{cases}
\Hom_R(P,P') & \text{if $s = s'$ and} \\
0 & \text{otherwise,} \\
\end{cases}
\end{equation}
}determined by sending a graded homomorphism from
$P \otimes_R R[U][s]$ to $P' \otimes_R R[U][s']$ to
the induced map on the degree $s$ pieces.
It follows that $\psi$ is fully faithful.

Given $F \in \cP^M(R[U])$,
the $M$-grading on $F$ gives a decomposition
$F = \bigoplus_m F_m$. If $m, m' \in M$ belong to different cosets of
  $U$, then $(R[U] \cdot F_m) \cap F_{m'} = 0$. Thus we have an
  internal direct sum decomposition
$$
F = \bigoplus_{s \in S} Q_s
$$
as $M$-graded $R[U]$-modules, where $Q_s = \bigoplus_{m \in
s + U} F_m$. 
Since $F$ is finitely
generated, $Q_s = 0$ for all but a finite number
of $s$. For each $s \in S$, we have 
$F_s \cong
Q_s \otimes_{R[U]} R$ (where $R[U] \to R$
is the augmentation map), and hence
$F_s$ is a finitely
generated and projective $R$-module.
If $m_1, m_2$ belong to the same coset
of $U$ in $M$, then $\chi^{m_2-m_1}: F_{m_1} \map{\cong} F_{m_2}$ is
an isomorphism of $R$-modules. Using this,
we see that the map
$$
F_s \otimes_R R[U][s] \to Q_s
$$
determined by $p \otimes \chi^u \mapsto \chi^u \cdot p$ is a graded
isomorphism of $R[U]$-modules. It follows that $F$ is isomorphic to
$\psi((F_s)_{s \in S})$, and hence $\psi$ is an equivalence.
\end{proof}

If $C, C'$ are $M$-graded rings, $\phi: C \to C'$ an $M$-graded ring
homomorphism and $F$ is an $M$-graded $C$-module, then the module
obtained from $F$ via extension of scalars along $\phi$, namely $C'
\otimes_C F$, acquires the structure of an $M$-graded $C'$-module
having the property that if $c' \in C'_{m_1}$ and $f \in F_{m_2}$ then
$c' \otimes f \in (C' \otimes_C F)_{m_1 + m_2}$ (see \cite[\S 2.4]
{MGR}). In particular, the module obtained from $C[m]$ by
extension of scalars along $\phi$ is $C'[m]$.

\begin{lemma}\label{EqSplitB}
The exact functor 
$$
\cP^M(R[U]) \to \cP^M(R[A])
$$
defined by extension of scalars 
induces a bijection on isomorphism
classes of objects. In particular, objects of $\cP^M(R[A])$ are
projective in the category of all $M$-graded $R[A]$-modules.
\end{lemma}

\begin{proof}
For a projective $R$-module $P$ and an arbitrary $M$-graded
$R[A]$-module $G$, 
we have 
\begin{equation} \label{E214}
\Hom_{\cP^M(R[A])}(P \otimes_R R[A][m], G)
\cong
\Hom_R(P, G_{m}).
\end{equation}
Since $G \mapsto G_m$ is an exact functor, $P \otimes_R R[A]$ is a
projective object in the category of all $M$-graded $R[A]$-modules.
In particular, 
the second assertion of the Lemma follows from the first one, using Lemma
\ref{EqSplitA}.

The $M$-graded $R$-algebra
map $R[U] \to R[A]$ is split by the $M$-graded $R$-algebra map $R[A]
\to R[U]$ defined by
$$
\chi^a \mapsto
\begin{cases}
\chi^a & \text{if $a \in U$ and} \\
0 & \text{if $a \notin U$.} \\
\end{cases}
$$
Since the composition $R[U] \into R[A] \onto R[U]$ is the identity,
the functor $\cP^M(R[U]) \to \cP^M(R[A])$ is split injective on
isomorphism classes of objects.

The proof of the surjectivity on isomorphism classes will use the
graded version of Nakayama's Lemma.
Let $I \subset R[A]$ denote the kernel of the split surjection $R[A]
\onto R[U]$  --- it is
generated as an $R$-module by $\{\chi^a \mid a \notin U\}$.
Clearly $I$ is $M$-graded and, moreover, every maximal $M$-graded ideal
of $R[A]$ contains $I$. Indeed, if $\fm$ is  a maximal $M$-graded ideal,
then $R[A]/\fm$ is a $M$-graded ring such that every non-zero homogeneous
element is a unit (and whose inverse is, necessarily, homogeneous).
For $a \notin U$, if $\ov{\chi^a} \ne 0$ in $R[A]/\fm$, then we would
have  $\ov{\chi^a} \cdot \ov{r \chi^b} = 1$
  for some $r \in R$ and $b \in A$. But then $a+b = 0$, contrary to $a
  \notin U$. Thus $\chi^a \in \fm$ for all $a \notin U$.
Since $I$ is contained in every maximal $M$-graded ideal,
the graded version of Nakayama's Lemma 
(see, for example, \cite[Theorem 3.6]{Perling-2004} for a proof) gives
us: If $G$ is a finitely
generated $M$-graded $R[A]$-module such that $IG = G$, then $G = 0$.

Given $E \in \cP^M(R[A])$, let $F = E \otimes_{R[A]} R[U] \in \cP^M(R[U])$ (with
the map $R[A] \to R[U]$ being the above split surjection) and let $\tilde{F}
= F \otimes_{R[U]} R[A]$.
We prove $E \cong \tilde{F}$ in
$\cP^M(R[A])$. As noted above, 
\eqref{E214} and Lemma \ref{EqSplitA} show that $\tilde{F}$ is a projective object
in the category of all 
$M$-graded $R[A]$-modules. Thus the canonical map $\tilde{F} \onto F$
lifts along the surjection $E \onto F$ to give a morphism $\theta:
\tilde{F} \to E$ in $\cP^M(R[A])$.
The map $\theta$ induces an isomorphism upon modding
out by $I$ and hence, by Nakayama's Lemma, $\coker(\theta) = 0$.
Since $E$ is projective as an ungraded $R$-module, the exact sequence
$$
0 \to \ker(\theta) \to \tilde{F} \to E \to 0
$$
remains exact upon application of $- \otimes_{R[A]} R[U]$, and hence,
using Nakayama's Lemma again, $\ker(\theta) = 0$.
\end{proof}

\begin{proof}[Proof of Theorem \ref{MainThm1}]
By Lemma \ref{EqSplitA}, we have 
$$
K_q^M(R[U]) \cong \bigoplus_S K_q(R) \cong K_q(R) \otimes_\Z \Z[M/U].
$$
In order to prove the theorem, it therefore suffices to 
prove the exact functor
\begin{equation} \label{E42c}
\cP^M(R[U]) \to \cP^M(R[A]),
\end{equation}
induced by extension of scalars, 
induces a homotopy equivalence on $K$-theory
spaces.

For any finite subset $F \subset S$, let $\cP_F^M(R[A])$ denote the
full subcategory of those objects in $\cP^M(R[A])$ isomorphic to one of the form
$$
\bigoplus_{i=1}^l P_i \otimes_R R[A][s_i]
$$
such that $s_i \in F$ for $i =1, \dots, l$. Define
$\cP_F^M(R[U])$ similarly. Note that $\cP_F^M(R[U])$ and
$\cP_F^M(R[A])$ are closed under direct sum and hence are exact subcategories.
Since $\cP^M(R[A]) = \varinjlim_{F \subset S} \cP^M_F(R[A])$ where $F$
ranges over all finite subsets of $S$ and since 
$K$-theory commutes with filtered colimits, it suffices to prove
$$
\cP^M_F(R[U]) \to \cP^M_F(R[A])
$$
induces an equivalence on $K$-theory for all finite $F \subset
S$. We proceed by induction on $\#F$. If $\#F = 1$, then by
\eqref{E814} and Lemma \ref{EqSplitB},
$\cP^M_F(R[U]) \to \cP^M_F(R[A])$ is an equivalence of categories.

Define a partial order $\leq$ on
$S$ by declaring $s \leq s'$ if and only if $s' - s \in A$. Then 
for projective $R$-modules $P, P'$ and elements $s, s' \in S$, we have
{\tiny
\begin{equation} \label{E42a} 
\Hom_{\cP^M(R[A])}(P \otimes_R R[A][s], P' \otimes_R R[A][s'])
\cong
\begin{cases} 
\Hom_R(P,P') & \text{if $s \leq s'$ and} \\
0 & \text{otherwise.}
\end{cases}
\end{equation}
}
Now assume $\# F > 1$ and 
let $s \in F$ be a maximal element. Define $F' = F
\setminus \{s\}$. We have a commutative diagram of exact functors
$$
\xymatrix{
\cP^M_{F'}(R[U]) \oplus \cP^M_{\{m\}}(R[U]) \ar[d] \ar[r] & 
\cP^M_{F'}(R[A]) \oplus \cP^M_{\{m\}}(R[A]) \ar[d] \\
\cP^M_{F}(R[U]) \ar[r] & 
\cP^M_{F}(R[A])
}
$$
in which the vertical maps are given by direct sum and the horizontal
maps are extensions of scalars. The left-hand vertical map and the top
horizontal map induce equivalences on $K$-theory  using
Lemma \ref{EqSplitA} and induction, respectively. It therefore suffices to
prove that the right-hand vertical map induces an equivalence on $K$-theory.
This follows from Waldhausen's generalization of the Quillen
Additivity Theorem, 
as we now explain.

Let $\cE$ denote the exact category consisting of short exact
sequences of objects of $\cP^M_F(R[A])$ of the form
\begin{equation} \label{E42d}
0 \to B \to P \to C \to 0
\end{equation}
with $B \in \cP^M_{\{m\}}(R[A])$ and $C \in \cP^M_{F'}(R[A])$. 
By Lemma \ref{EqSplitB}, for any such short exact sequence, we have that
$P$ is isomorphic to $B \oplus C$. Moreover, 
by \eqref{E42a} there are no non-trivial maps from $B$ to $C$, and
hence this exact sequence is isomorphic to 
$$
0 \to B \map{\cvec{1}{0}} B \oplus C \map{(0,1)} C \to 0.
$$
Thus $\cE$ is equivalent to the full subcategory consisting of such
``trivial'' exact sequences. A morphism from one such exact sequence
to another is completely determined by the map on middle objects.
That is, the functor
$\cE \to \cP^M_F(R[A])$ sending the exact sequence \eqref{E42d} to $P$
is an equivalence of categories.  On the other hand, 
Waldhausen's Additivity Theorem \cite{Wald} shows 
that the functor
$$
\cE \to \cP^M_{\{m\}}(R[A])  \oplus 
\cP^M_{F'}(R[A]) 
$$
sending \eqref{E42d} to $(B, C)$ induces an equivalence on $K$-theory.
\end{proof}

\section{The Equivariant $K$-theory of Affine Toric Varieties}

In this section we provide an interpretation of Theorem \ref{MainThm1} for
toric varieties.

We adopt the notational
conventions for toric varieties found in Fulton's book \cite{Fulton}. An affine toric variety 
is defined from a
strongly convex rational polyhedral cone $\s$ in $N \otimes_\ints \real$ where
$N \cong \ints^n$ is an $n$ dimensional lattice. Let $M:=\Hom_\ints (N,
\ints) \isom \ints^n$ be the dual lattice and define
the
dual cone of $\s$ by
$$
\s^\dual:=\{ u \in M
\tensor_\ints \real | u(v) \ge 0 \mbox{ for all } v
\in \s \}.
$$
We have
that $\s^{\dual} \intersect M$ is a finitely generated abelian monoid,
by Gordan's Lemma, and
hence, for any commutative ring $R$, the corresponding monoid ring
$R[ \s^{\dual} \intersect M]$ is a finitely generated $R$-algebra.
We let 
$$
U_{\sigma, \Z} = \Spec \Z[\s^\vee \cap M],
$$ 
the {\em affine toric scheme} over $\Z$ associated to $\s$.

Note that for any commutative ring $R$, we have 
$$
U_{\sigma, R} := U_{\sigma, \Z}
\times \Spec R = \Spec R[\s^\vee \cap M].
$$ 
In particular, for a field $k$, the affine $k$-variety
$U_{\sigma,k} 
= \Spec k[\s^\vee \cap M]$ 
is the classical toric $k$-variety associated to $\sigma$.

For any commutative ring $R$, the $R$-algebra $R[\s^\vee \cap M]$ is an $M$-graded
$R$-algebra,  and this grading amounts to an action of the $n$-dimensional
torus scheme $T := \Spec \Z[M]$ on $U_{\s, R}$. Viewing $U_{\s, R}$ as
$U_{\s, \Z} \times \Spec R$, the action of $T$ is given by the
usual action on $U_{\sigma, \Z}$ and the trivial action $\Spec R$.  An
equivariant vector bundle over $U_{\s, R}$ is identified as a
projective module over $R[\s^\vee \cap M]$ that is $M$-graded.  We
therefore obtain
$$
K_*^{M}(R[\s^\vee \cap M]) \cong K_*^T(U_{\s, R}).
$$

Finally, observe that $U(\sigma^\vee \cap M) = \sigma^\perp \cap M$,
and we define $M_\sigma := M/(\sigma^\perp \cap M)$.
The following is 
thus an immediate consequence of 
Theorem
\ref{MainThm1}.

\begin{theorem} \label{Thm2} 
For any commutative ring $R$ and strongly convex rational cone $\s$,
there is a natural isomorphism
$$
K^T_q(U_{\sigma,R}) \cong
K_q(R) \otimes_\Z \Z[M_\sigma].
$$
\end{theorem}

In particular, we see that Equation \eqref{E1} holds for any affine
toric variety, not only the smooth ones. Observe that $M_\sigma$, as
just defined, coincides with the group of characters on the minimal
orbit of $U_\sigma$.

\begin{remark}
The isomorphism of Theorem \ref{MainThm1} is natural in $R$ in the obvious sense
and 
is natural in $A$ 
in the following sense: If $A \subset
A' \subset M$ is an inclusion of submonoids of $M$,
then
$$
\xymatrix{
K_q(R) \otimes_\Z \Z[M/U(A)] \ar[r]^{\phantom{XXX} \cong} \ar[d]&
K_q^{M}(R[A])  \ar[d] \\
K_q(R) \otimes_\Z \Z[M/U(A')]) \ar[r]^{\phantom{XXX} \cong} &
K_q^{M}(R[A'])  \\
}
$$
commutes, where the left-hand map is the canonical quotient map
and the right-hand map is induced by extension of scalars.

Consequently, the isomorphism of Theorem \ref{Thm2} is natural in $R$
and with respect to 
the inclusion of a  face $\tau$ into $\sigma$. In the
latter case, the map 
$$
K^T_q(U_{\sigma, R}) \to
K^T_q(U_{\tau, R}) 
$$
is induced by pullback along the equivariant open immersion
$U_{\tau, R} \subset U_{\sigma, R}$ and the map 
$$
K_q(R) \otimes_\Z \Z[M_\sigma] \to
K_q(R) \otimes_\Z \Z[M_\tau]
$$
is the map induced by the canonical surjection $M_\sigma \onto M_\tau$.
\end{remark}

\section{The Vezzosi-Vistoli Theorem}

In this section, we use \eqref{E1} from the introduction, the
theory of sheaves on fans and the foundational results of Thomason
\cite{Thomason} concerning equivariant $K$-theory
to recover a result due to Vezzosi and
Vistoli \cite{VV,VV-Erratum}: For a field $k$ and a
smooth toric $k$-variety $X=X(\D)$
defined by a fan $\D$, the sequence
$$
0\lra K^T_q(X) \lra \bigoplus_{\s\in Max(\D)} K^T_q(U_\s) \buildrel
{\partial} \over\lra \bigoplus_{\d,\t \in Max(\D), \d < \t} K^T_q(U_{\d\cap
  \t})
$$
is exact. Here, $Max(\D)$ is the set of maximal cones in $\D$ and we
choose, arbitrarily, a total ordering for this set. 
The map ${\partial}$ is given as
follows.  For $f=(f_{\s})_{\s\in Max(\D)}$ in $\bigoplus_{\s\in
  Max(\D)} K^T_q(U_\s)$, the $(\d < \t)$-component
of its image
is $f_{\t}|_{U_{\d\cap \t}}-f_{\d}|_{U_{\d\cap
  \t}}\in K^T_q(U_{\d\cap \t})$.

In fact, we prove that the sequence
\begin{equation} \label{VVseq}
0\ra K^T_q(X) \ra \bigoplus_{\s} K^T_q(U_\s)\ra
\bigoplus_{\d < \t} K^T_q(U_{\d\cap \t}) \ra
\bigoplus_{\d < \t < \epsilon} K^T_q(U_{\d\cap \t \cap \epsilon}) \ra \cdots
\end{equation}
is exact, where
$\bigoplus_{\s} K^T_q(U_\s)\ra \bigoplus_{\d < \t}
K^T_q(U_{\d\cap \t}) \ra\cdots$ is the \v Cech complex of the presheaf
$K^T_q$ for the equivariant open cover
$\sV = \{U_\s \mid  \text{$\s$ is a maximal cone in
$\D$} \}$. Using Equation \eqref{E1} (or our Theorem \ref{Thm2}), the
exactness of this sequence is equivalent to the existence of an exact
sequence of the form 
{
\begin{equation} \label{VVseq2}
0\ra K^T_q(X) \ra \bigoplus_{\s} K_q(k) \otimes_\Z \Z[M_\s]
\ra
\bigoplus_{\d < \t} K_q(k) \otimes_\Z \Z[M_{\d \cap \t}]
\ra\cdots.
\end{equation}}

We define a topology on the finite
set of cones comprising the fan $\D$ by declaring the open subsets to
be the subfans of $\D$; see \cite{BB} or \cite{BL}. In other words,
we view $\D$ as a poset via face containment, $\prec$, 
and we give $\D$ the ``poset topology'', in which an open
subset $\Lambda$ is a subset satisfying the condition what whenever $x \prec
y$ and $y \in \Lambda$, we have $x \in \Lambda$.
For a cone
$\sigma \in \D$, let $\langle \sigma \rangle$ denote the fan
consisting of $\sigma$ and all its faces (i.e., the smallest
open subset of $\D$ containing $\s$). Observe that for a sheaf $\cF$
on $\D$, we have $\cF(\langle \sigma \rangle) = \cF_\sigma$, the stalk
of $\cF$ at the point $\sigma$.

For this topology, sheaves are uniquely determined by their stalks
and the maps between their stalks arising from comparable elements of
the poset (see \cite[\S 4.1]{Baggio}).  That is, there is an equivalence between the category of
contravariant functors from the poset $\D$ to the category of abelian
groups and the category of sheaves of abelian groups on the
topological space $\D$.  (Recall that a poset may be viewed as a
special type of category.)  Given a sheaf $\cF$ on the space $\D$, the
associated functor on the poset $\D$ sends $\sigma \in \D$ to
$\cF_{\sigma} = \cF(\langle \sigma \rangle)$ and sends a face
inclusion $\tau \prec \sigma$ to the map induced by
$\langle \tau \rangle \subset
\langle \sigma \rangle$.
Given a contravariant functor $F$ on the
poset $\D$, the value of associated sheaf $\cF$ on an open subset
$\Lambda$ of $\D$ is given by
$$
\cF(\Lambda)= \varprojlim_{\sigma \in \Lambda} F(\sigma).
$$


\begin{theorem} \label{ExactSeq}
Assume $X = X(\Delta)$ is a smooth toric variety defined over
  an arbitrary field $k$. Then the presheaf 
$\Lambda \mapsto K_q^T(X(\Lambda))$ defined on $\D$ is a flasque
sheaf. Moreover, there is an isomorphism 
$$
K_q^T(X) \cong K_q(k) \otimes K_0^T(X).
$$
and the sequences \eqref{VVseq} and \eqref{VVseq2} are exact.
\end{theorem}

\begin{proof} 
Let $\cA_q$ be the sheaf on $\D$ associated to the functor
sending a cone $\s$ to $K_q(k) \otimes \Z[M_\s]$ and a face inclusion 
$\tau \prec \sigma$ to the map induced by the canonical quotient $M_\s
\onto M_\tau$. 

The sheaf $\cA_0$  is flasque by 
  \cite{Baggio}.
Since $\cA_0$ is a flasque sheaf
of torsion free abelian groups, the presheaf $K_q(k) \otimes_\Z \cA_0$
is actually a sheaf. Indeed, for any open subset $U$ and open covering
$U = \cup_i V_i$ of it, the map from $\cA_0(U)$ to the 
associated \v Cech complex is a quasi-isomorphism by 
\cite[III.4.3]{Hartshorne}, and since $\cA_0$ is torsion free, this
map remains a quasi-isomorphism upon tensoring by any
abelian group.
It now follows from the correspondence between functors and sheaves that
$\cA_q \cong K_q(k) \otimes \cA_0$. In particular, $\cA_q$ is also
flasque.

For a subfan $\Lambda$ of $\D$, 
let $\sV$ be the Zariski open covering $\{U_\s \mid \text{$\sigma$ is a
  maximal cone in $\Lambda$} \}$ 
of $X(\Lambda)$ and let
$\sU$ be
the open covering $\{ \langle \s \rangle \mid \s \in Max(\D)\}$
of $\Lambda$.
By Equation \eqref{E1} (or Theorem \ref{Thm2}), the 
\v Cech cohomology complex of the presheaf $\cK_q^T(-)$ on $X(\Lambda)$ for 
the open covering $\cV$ coincides with the \v Cech cohomology complex of the
sheaf $\cA_q$ for the open covering $\sU$.
Since the higher \v Cech cohomology
of flasque sheaves vanishes
\cite[III.4.3]{Hartshorne}, we have
\begin{equation} \label{E415}
\check{H}^p\l(\sV, K_q^T\r) = \check{H}^p\l(\sU, \cA_q\r) = 
0, \, \, \text{for all $p > 0$.}
\end{equation}

Thomason \cite{Thomason} has proven that $\cK^T$ coincides
with equivariant $G$-theory (defined from equivariant coherent
sheaves) and that the latter satisfies the usual
localization property relating $X$, an equivariant closed subscheme,
and its open complement. From this one deduces that if $X(\Lambda) = U \cup
V$ is covering by equivariant open subschemes, then
$$
\xymatrix{
\cK^T(X(\Lambda)) \ar[r] \ar[d] & \cK^T(U) \ar[d] \\
\cK^T(V) \ar[r] & \cK^T(U \cap V) \\
}
$$
is a homotopy cartesian square. Arguing just as in \cite[\S 8]{TT}, one
obtains
a convergent spectral  sequences
$$
\check{H}^p\l(\sV, K_q^T\r) \Lra K_{q-p}^T(X(\Lambda)).
$$
Using \eqref{E415}, this spectral sequence collapses to give
\begin{equation} \label{E415b}
\check{H}^0\l(\sV, K_q^T\r) \isom K_{q}^T(X(\Lambda)), \, \, \text{for all
  $q$.}
\end{equation}

Combining \eqref{E415b} and \eqref{E415}
gives that the complexes 
$$
0 \to K_q^T(X(\Lambda))
\ra \bigoplus_{\s} K^T_q(U_\s)\ra
\bigoplus_{\d < \t} K^T_q(U_{\d\cap \t}) \ra
 \cdots
$$
and
$$
0\ra \cA_q(\Lambda)
\ra \bigoplus_{\s} K_q(k) \otimes_\Z \Z[M_\s]
\ra
\bigoplus_{\d < \t} K_q(k) \otimes_\Z \Z[M_{\d \cap \t}]
\ra\cdots
$$
are exact and isomorphic to each other.
In particular, $\Lambda \mapsto \cK_q^T(X(\Lambda))$ is isomorphic to
the flasque sheaf
$\cA_q$.

The remaining assertions of the Theorem follow immediately.
\end{proof}

\thanks{
{\bf Acknowledgements:}
We are very
  grateful to Dan Grayson for suggesting the use of the 
  Additivity Theorem in the proof of Theorem \ref{MainThm1}. This
  allowed us to simplify our original proof and to make the result
  more general.

We are also indebted to the anonymous referee for many thoughtful
suggestions which greatly improved the exposition of this paper. 
}

\bibliographystyle{plain}


\end{document}